\documentclass[conference]{IEEEtran}
\IEEEoverridecommandlockouts

\usepackage{cite}
\usepackage{amsmath,amssymb,amsfonts, amsthm}
\usepackage{algorithmic}

\usepackage{graphicx}
\usepackage{textcomp}
\usepackage{xcolor}
\usepackage{hyperref}

\newtheorem{theorem}{Theorem}
\newtheorem{lemma}{Lemma}
\newtheorem{algorithm}{Algorithm}

\newcommand{\R}{\mathbb R}
\newcommand{\C}{\mathbb C}

\newcommand{\mcM}{\mathcal M}

\newcommand{\Sptn}{\operatorname{Sp}(2n,\R)}
\newcommand{\sptn}{\mathfrak{sp}(2n,\R)}
\newcommand{\shtn}{\mathfrak{sh}(2n,\R)}
\newcommand{\Sptp}{\operatorname{Sp}(2p,\R)}
\newcommand{\sptp}{\mathfrak{sp}(2p,\R)}
\newcommand{\shtp}{\mathfrak{sh}(2p,\R)}
\newcommand{\SpStnp}{\operatorname{SpSt}(2n,2p)}

\newcommand{\sym}{\operatorname{sym}}

\DeclareMathOperator{\diag}{diag}
\makeatletter
\renewcommand*\env@matrix[1][*\c@MaxMatrixCols c]{%
	\hskip -\arraycolsep
	\let\@ifnextchar\new@ifnextchar
	\array{#1}}
\makeatother

\def\BibTeX{{\rm B\kern-.05em{\sc i\kern-.025em b}\kern-.08em
    T\kern-.1667em\lower.7ex\hbox{E}\kern-.125emX}}
\begin{document}

\title{A polar-factor retraction on the symplectic Stiefel manifold with closed-form inverse*\\
\thanks{This work was supported by the Independent Research Foundation Denmark, DFF, grant nr. 3103-00094B.}
}

\author{\IEEEauthorblockN{Ralf Zimmermann}
\IEEEauthorblockA{\textit{Department of Mathematics and Computer Science} \\
\textit{University of Southern Denmark}\\
Odense, Denmark \\
https://orcid.org/0000-0003-1692-3996}
}

\maketitle

\begin{abstract}
In Riemannian computing applications, it is crucial to map manifold data to a Euclidean domain, where vector space arithmetics are available, and back. Classical manifold theory guarantees the existence of such mappings, called charts and parameterizations, or, collectively, local coordinates. When computational efficiency is of the essence, practitioners usually resort to retraction maps to define local coordinates. Retractions yield first-order approximations of the Riemannian normal coordinates.

 This work introduces a new retraction on the symplectic Stiefel manifold that features a closed-form inverse. 
 We expose essential features and compare the numerical performance to a selection of existing retractions.
 To the best of our knowledge, prior to this work, the so-called  Cayley retraction was the only retraction on the symplectic Stiefel manifold with known closed-form inverse.
\end{abstract}


\section{Introduction}
Practical data processing on manifolds requires local coordinates, which make it possible to map data `there and back':
‘There’ means mapping data from a Euclidean coordinate domain, for example the tangent space, to the manifold.
‘Back’ refers to the reverse action of mapping manifold data into a coordinate domain.

Let $\mcM$ be a Riemannian manifold and $x\in\mcM$.
The canonical Euclidean coordinate domain associated with $\mcM$ at $x$ is the tangent space $T_x\mcM$.
According to \cite[Ch. 4]{AbsilMahonySepulchre2008}, \cite[Sec. 3.6]{Boumal2023}, a retraction $R$ on $\mcM$ is a smooth mapping from the tangent bundle $T\mcM$ to $\mcM$ such that (1) for all $x\in\mcM$, $R_x: T_x\mcM \to \mcM$,\footnote{The concept works the same, if all $R_x$ are defined on open subsets around $0$ in the associated tangent spaces $T_x\mcM$, $x\in \mcM$.} (2) $R_x(0)=x$, where $0$ is the zero of $T_x\mcM$, and (3) the differential of $R_x$ at $0$ is the identity, $d(R_x)_0 = \mathrm{id}_{T_x\mcM}$.\\
The latter property guarantees the existence of a local inverse 
\[
	(R_x)^{-1}: \mcM\supset \mathcal{D} \to T_x\mcM,
\]
obviously defined only on a suitable small domain $\mathcal{D}\subset \mcM$ around $x$.

In certain applications, e.g., in Riemannian optimization, only the forward retraction is needed, but for manifold interpolation or for computing Riemannian averages, there must also be an efficient way to compute $(R_x)^{-1}$.\\
Moreover, $(R_x)^{-1}$ gives rise to a vector transport from one tangent space to another, because its differential is a map
\[
	d(R_x^{-1})_y: T_y\mcM \to T_{R_x^{-1}(y)}\mcM.
\]

In this work, we consider the symplectic Stiefel manifold, which is formally introduced in the next section.
This matrix manifold features in a variety of applications often related to physical problems, see, e.g.,  \cite{GaoSonAbsilStykel2020symplecticoptimization, GauSonStykel:2024, OviedoHerrera:2023, BendokatZimmermann:2021}.
All of these works include a selection of retractions on the symplectic Stiefel manifold, with the essential ones being the
Cayley retraction \cite{GaoSonAbsilStykel2020symplecticoptimization, BendokatZimmermann:2021, OviedoHerrera:2023}, the SR retraction \cite{GauSonStykel:2024},  the Riemann and pseudo-Riemann exponentials of \cite{BendokatZimmermann:2021} and the quasi-geodesics of \cite{GaoSonAbsilStykel2020symplecticoptimization}.
To the best of our knowledge, only the Cayley retraction was shown to feature a closed for inverse, see \cite{BendokatZimmermann:2021}.
With this work, we add a second one to the list of symplectic Stiefel retractions with closed form inverse.

Symplectic and Hamiltonian matrices are also a classical subject of numerical linear algebra studies, \cite{Fassbender:1999hamiltonian, Kressner:2005, BennerKressnerMehrmann:2007}.

\section{Background: Symplectic and Hamiltonian matrices}
We define the structure matrices.
The $(n\times n)$-identity matrix is denoted by $I_n$.
Further,
\[
J_n=\begin{bmatrix}
	0 & I_n\\
	-I_n& 0
\end{bmatrix}, \hspace{0.1cm}
I_{n,p}:= \begin{bmatrix} I_p \\ 0 \end{bmatrix},
\hspace{0.1cm}
E_{n,p}:=\begin{bmatrix} I_{n,p} & 0 \\ 0 & I_{n,p} \end{bmatrix},
\]
with $J_n\in \R^{2n \times 2n}$, $I_{n,p}\in \R^{n \times p}$,
$E_{n,p}\in \R^{2n \times 2p}$.
It holds $J_{n}^T = -J_{n}=J_{n}^{-1}$.\\
The \emph{standard symplectic vector space} is $(\R^{2n},\omega_0)$ with bilinear form
\begin{equation*}
	\omega_0(x,y):=x^TJ_{n}y, \quad  
	x,y\in\R^{2n},
\end{equation*} 
called the \emph{standard symplectic form}.

For any, possibly rectangular matrix $A \in \R^{2n \times 2p}$ the \emph{symplectic inverse} \cite{PengMohseni2016}
is
\begin{equation*}
	A^+:= J_{2p}^TA^TJ_{2n}.
\end{equation*} 
The symplectic group is
\[
\Sptn = \{M\in\R^{2n\times 2n} \mid M^TJ_nM = J_n\}. 
\]
Its Lie algebra is the set of Hamiltonian matrices
\[
\sptn = \{\Omega\in\R^{2n\times 2n} \mid 
(J_n\Omega)^T = J_n\Omega\}.
\]
Hamiltonian matrices have the block structure
\begin{equation}
	\label{eq:Ham_block}
\Omega = \begin{bmatrix}
	A&B\\
	C& -A^T
\end{bmatrix},\quad B,C\in\sym(n).
\end{equation}
In particular, for $\Omega\in\sptn$,
\[
\Omega^TJ_n = -J_n\Omega,
\quad \Omega^+=J_n^T\Omega^TJ_n = -\Omega
\]
The set of skew-Hamiltonian matrices is denoted by
\[
\shtn = \{\Omega\in\R^{2n\times 2n} \mid (J_n\Omega)^T = -J_n\Omega\}.
\]
The rectangular relative of the symplectic group is the \emph{real symplectic Stiefel manifold}
\begin{align*}
	\SpStnp&:=\left\lbrace U \in \R^{2n \times 2p}\ \middle|\ U^+U=I_{2p} \right\rbrace \\
	&= \left\lbrace U \in \R^{2n \times 2p}\ \middle|\ U^TJ_{n}U=J_{p} \right\rbrace.
\end{align*}
It contains the matrices $U \in \R^{2n\times 2p}$, whose column vectors form symplectic bases for the $2p$-dimensional symplectic subspaces of $(\R^{2n},\omega_0)$ and was treated in \cite{GaoSonAbsilStykel2020symplecticoptimization,GaoSonAbsilStykel2021symplecticEuclidean,SonAbsilGaoStykel2021symplecticeigenvalue}.\\
It follows that for any $U\in \SpStnp$, $UU^+$ is a projector, since $UU^+UU^+ = UI_pU^+ = UU^+$.
This makes $I_{2n} - UU^+$ the complementary projector.\\
By \cite[Prop. 3.2]{BendokatZimmermann:2021}, the tangent space at $U \in \SpStnp$ is given by
\begin{equation}
	\label{eq:TangentSpaceStiefel}
		T_U\SpStnp 
		=\left\lbrace D \in \R^{2n\times 2p}\ \mid \ U^+D \in \sptp\right\rbrace.
\end{equation}
%
%
%

%
\section{A symplectic polar factor retraction}
In this section, we construct an explicit set of parameterizations/coordinate charts on $\SpStnp$ given by a retraction and its explicit inverse.
The construction can be considered as the symplectic counterpart to the polar-light retraction on the classical Stiefel manifold that has recently been proposed in \cite{JensenZimmermann:2026}. Where the polar-light retraction relies on the standard polar-factor decomposition \cite[Thm. 8.1]{Higham:2008:FM}, here we employ the  symplectic polar decomposition 
of \cite[Theorem 2.5]{Teretenkov:2022}, which, in turn relies on the theory developed in \cite{Fassbender:1999hamiltonian}.
This theorem states that any $X\in \R^{2p\times 2p}$ can be decomposed into
\begin{equation}
	\label{eq:sympPolarDecomp}
	X = MS, \quad M\in\shtn, \quad SJ_{p}S^T = J_p,
\end{equation}
if $Y = XX^+ (=-XJX^TJ)$ has no zero or negative real eigenvalues.
Observe that the theorem gives
\begin{equation}
	\label{eq:sympPolarDecomp1}
	X^T=S^TM^T, \text{ with } \quad S^T\in \Sptp.
\end{equation}

\subsection{The forward symplectic polar-light retraction}

Let $U\in \SpStnp$ be a fixed base point and let  $D\in T_U\SpStnp$.
\paragraph{Ansatz} 
As with the polar-light retraction \cite{JensenZimmermann:2026} on the classical Stiefel manifold,
we aim for a retraction
\begin{align}
	\nonumber
	R^f_U: & T_U\SpStnp \to \SpStnp\\
		\label{eq:def_Rf}
	&D\mapsto \left(U\exp_m(U^+D) + (I-UU^+)D\right)G_f.
\end{align}
The matrix factor $G_f\in\R^{2p\times 2p}$ is to be chosen as a `symplectifier' that ensures that 
$R^f_U(D) \in \SpStnp$.\footnote{A similar construction with two matrix factors $R,S$ is considered in \cite[eq. (6)]{OviedoHerrera:2023},
with $R$ in place of $\exp_m(U^+D)$ and $S$ in place of $G_f$. It is clear $R$, $S$ have to satisfy certain compatibility conditions to produce a valid retraction.}
Writing $\Delta= U\exp_m(U^+D) + (I-UU^+)D$, we require
\[
	G_f^T \Delta^TJ_n\Delta G_f = J_p.
\]
Minding that $U^+D\in\sptp$ and that the matrix exponential of a Hamiltonian matrix is a symplectic matrix,
a calculation shows
\begin{equation}
\label{eq:DeltaTJDelta}
\Delta^TJ_n\Delta = J_p + D^TJ_pD - (U^+D)^T J_p (U^+ D).
\end{equation}
This is a skew-symmetric matrix which for $D, U$ of small norm has full rank. The idea is now to perform a "mirror-splitting" into matrix factors with $J_p$ in the center.
This is obtained as follows: A real Schur decomposition yields
\[
\Delta^TJ_n\Delta = \Phi\Lambda_S\Phi^T,\quad
\Lambda_S = 
\begin{pmatrix}[cc|c|cc]
	0           & \lambda_1 &         &   & \\
	-\lambda_1  & 0         &         &   &  \\
	\hline
	&           & \ddots  &   &     \\
	\hline
	&           &         & 0 & 
	\lambda_p\\
	&           &         &-\lambda_p &0 
\end{pmatrix},
\]
with $\Phi\in O(2p)$.
Mind that the actual eigenvalues are not the $\lambda_j's$ but the  purely imaginary, complex conjugate pairs $\pm i\lambda_j\in \C\setminus\R$. 
Because of the full rank, after possibly reordering, we may assume that  $\lambda_j> 0$ for all $j=1,\ldots,p$.
Write $\Lambda_D = \diag(\lambda_1,\lambda_1,\ldots,\lambda_p,\lambda_p)\in \R^{2p\times 2p}$.
Then
\[
\Lambda_S = \Lambda_D^{\frac12}
\begin{pmatrix}[cc|c|cc]
	0           & 1 &         &   & \\
	-1  & 0         &         &   &  \\
	\hline
	&           & \ddots  &   &     \\
	\hline
	&           &         & 0 & 
	1\\
	&           &         &-1 &0 
\end{pmatrix}\Lambda_D^{\frac12}
=: \Lambda_D^{\frac12}I_S\Lambda_D^{\frac12},
\]
There is a permutation matrix $\Pi$ such that
$
\Pi^T I_S\Pi= J_p
$ so that we obtain the final splitting
\[
	\Delta^TJ_n\Delta = \left(\Phi\Lambda_D^{\frac12}\Pi\right) J_p \left(\Pi^T \Lambda_D^{\frac12}\Phi^T\right)
	=:X J_p X^T.
\]
Any matrix factor $G_f$ with $G_f^T (X J_p X^T) G_f = J_p$ in \eqref{eq:def_Rf} will guarantee that 
$R_f(D)\in \SpStnp$. Obviously, the choice is not unique.
Our proposal is to polar-decompose $X$ according to \cite[Theorem 2.5]{Teretenkov:2022}
\[
	X = \bigl((XX^+)^{\frac12}\bigr) \bigl((XX^+)^{-\frac12} X\bigr) =: M S.
\]
Via this decomposition, the `symplectic part' of $X$ is factored out, because
$S^T$ is symplectic and  $X J_p X^T = M (S^T)^T J_p S^T M^T =  M J_p M^T$. Hence, $S$ needs not be formed in practice, and for $G_f$, we choose
\begin{equation}
	\label{eq:Gf}
	G_f = (M^T)^{-1} =  \bigl((X^T)^+ X^T)^{-\frac12},
\end{equation}
where the square root is taken to be the principal primary matrix square root function \cite{Higham:2008:FM},
which inherits the skew-Hamiltonian structure of $(X^T)^+ X^T$ and  $\left((X^T)^+ X^T\right)^{-1}$.
Finally, observe that 
\[
	XX^+ = -XJ_pX^TJ_p = - \Delta^T J_n\Delta J_p.
\]
This shows that (1) $X$ needs not be formed in practice and (2) any ambiguity in the Schur decomposition has no impact on the final matrix $G_f$.
The definition of $G_f$ completes the construction of the forward symplectic polar factor retraction \eqref{eq:def_Rf}.
Algorithm \ref{alg:Rf_U} summarizes the procedure.
\begin{algorithm}[Numerical evaluation of $R^f_U$]
\begin{algorithmic}[1]
	\REQUIRE{$U\in \SpStnp$, $D\in T_U\SpStnp$.}
	\STATE{$K\leftarrow J_p + D^TJ_pD - (U^+D)^T J_p (U^+ D)$ \textcolor{gray}{$(=\Delta^TJ_n\Delta)$.}
	}
	\STATE{$ M \leftarrow\sqrt{- KJ_p}$. \textcolor{gray}{\small(principal primary square root)}}
	\STATE{$G_f \leftarrow M^{-T}$.}
	\ENSURE{$R^f_U(D) \leftarrow \left(U\exp_m(U^+D) + (I-UU^+)D\right)G_f$}
\end{algorithmic}
\label{alg:Rf_U}
\end{algorithm}
To prevent stability issues and to increase efficiency, in practice, the matrix exponential $\exp_m(U^+D)$
is replaced by the Cayley transformation, which also sends Hamiltonian matrices to symplectic matrices.
\subsection{The reverse symplectic polar-light retraction}
Let $U, \tilde U\in \SpStnp$, with $U$ acting as the base point.
We construct a map $R^\iota_U$ that maps relative open neighborhoods on $\SpStnp$
to open neighborhoods on $T_U\SpStnp$. The map is constructed to be the local inverse to $R^f_U$.
\paragraph{Ansatz} 
Mimicking again the polar-light retraction of \cite{JensenZimmermann:2026},
we set
\begin{align}
	\nonumber
	R^\iota_U: & \SpStnp \to T_U\SpStnp \\
		\label{eq:def_Ri}
	&\tilde U\mapsto U\log_m(U^+\tilde U G_r) + (I-UU^+)\tilde U G_r.
\end{align}
With the choice to employ the matrix logarithm made, we are only left to choose
the matrix factor $G_r\in\R^{2p\times 2p}$, which is to be chosen as a `Hamiltonifier' that ensures that 
$R^\iota_U(\tilde U)\in \SpStnp$. This means that $U^+D$
Writing $D= U\log_m(U^+\tilde U G_r) + (I-UU^+)\tilde U G_r$, and keeping in mind that $UU^+$ and $I_{2n}-UU^+$ are projectors,
we only require that $U^+\tilde U G_r\in \Sptp$, i.e.,
\[
\left(G_r^T(U^+\tilde U)^T\right) J_p  \left((U^+\tilde U) G_f\right) = J_p.
\]
In contrast to the forward construction, the mirror-splitting with $J_p$ in the center is immediate.
We apply the symplectic polar decomposition of \cite[Theorem 2.5]{Teretenkov:2022} to obtain
\[
	(U^+\tilde U)^T  = N T, \quad N\in \shtp, T^T\in \sptp,
\]
 where $N = \left((U^+\tilde U)^T\left((U^+\tilde U)^T\right)^+\right)^{\frac12}$.
Then, set
\[
	G_r = (N^T)^{-1}.
\]
This completes the definition of \eqref{eq:def_Ri}.
The numerical procedure is summarized in Alg. \ref{alg:Ri_U}.
\begin{algorithm}[Numerical evaluation of $R^\iota_U$]
\begin{algorithmic}[1]
	\REQUIRE{$U, \tilde U\in \SpStnp$.}
	\STATE{$H\leftarrow (U^+\tilde U)^T\left((U^+\tilde U)^T\right)^+$.}
	\STATE{$N \leftarrow\sqrt{H}$. \textcolor{gray}{\small(principal primary square root)}}
	\STATE{$G_r \leftarrow N^{-T}$.}
	\ENSURE{$R^\iota_U(\tilde U) \leftarrow U\log_m(U^+\tilde UG_r) + (I-UU^+)\tilde UG_r$}
\end{algorithmic}
\label{alg:Ri_U}
\end{algorithm}
To prevent stability issues and to increase efficiency, in practice, the matrix logarithm $\exp_m(U^+\tilde U G_r)$
is replaced by the inverse Cayley transformation, which also sends symplectic   matrices to Hamiltonian matrices.
\subsection{The maps of \eqref{eq:def_Rf} and \eqref{eq:def_Ri} are inverse to each other}
Given $U, \tilde U\in \Sptn$, with $U$ acting as the base point, we show that 
$R^f_U\circ  R^\iota_U = \mathrm{id}\vert_{\mathcal{U}}$, where $\mathcal{U}$ is a suitable open domain around $U$ such that the 
composite map is well defined.
It holds
\begin{align*}
	R^f_U\circ  R^\iota_U (\tilde U) &=R^f_U\left(U\log_m(U^+\tilde UG_r) + (I-UU^+)\tilde UG_r\right)\\
	&= \biggl(U\exp_m\log_m(U^+\tilde UG_r) \\
	& \qquad + (I-UU^+)\tilde UG_r\biggr)G_f\\
	&= \left(UU^+\tilde U + (I-UU^+)\tilde U\right)G_rG_f= \tilde UG_rG_f.
\end{align*}
This remains true, if $\exp_m(X), \log_m(Y)$ are replaced by their Cayley counterparts
$ \mathrm{Cay}(X) = (I - \frac12X)^{-1} (I + \frac12X)$, $ \mathrm{Cay}^{-1}(Y) = 2(Y+I)^{-1}(Y-I)$.
Hence, we are only required to check that $G_r=G_f^{-1}$ under the composition $R^f_U\circ  R^\iota_U$.
With $D=R^\iota_U (\tilde U)$ as the input argument to $R^f_U$ the same calculation as for \eqref{eq:DeltaTJDelta}
shows
\[
	\Delta=\tilde UG_r, \quad \Delta^TJ_n\Delta = G_r^TJ_pG_r=K.
\]
According to the procedure, $G_f = \left((-KJ_p)^{-\frac{1}{2}} \right)^T$.
By construction, $G_r\in \shtp$, i.e., $(J_pG_r)^T = -J_pG_r$ and we additionally have $G_r^{-1}\in \shtp$.
Hence,
\[
	-KJ_p = -G_r^TJ_pG_rJ_p = (J_pG_r)^T G_r J_p = J_p^TG_r^2J_p.
\]
Similarity transformations factor out of matrix functions, so that
\begin{align*}
	G_f &=  \left((-KJ_p)^{-\frac{1}{2}} \right)^T = \left((J_p^TG_r^2J_p)^{-\frac{1}{2}} \right)^T\\
	&= \left(J_p^TG_r^{-1}J_p \right)^T = J_p^T( -(J_pG_r^{-1})^T)  \\
	&= J_p^T(J_pG_r^{-1}) = G_r^{-1}.
\end{align*}
The proof that locally, $R^\iota_U \circ R^f_U  = \mathrm{id}\vert_{\mathcal{D}}$, where 
$\mathcal{D}$ is a suitable small open domain around $0$ in $T_U\SpStnp$ is analogous.
In total, we have proven
\begin{lemma}
	\label{thm:invertible}
	For $U\in \SpStnp$ the maps $R^f_U$ and $R^\iota_U$ of \eqref{eq:def_Rf} and \eqref{eq:def_Ri}, respectively,
	are invertible and inverse to each other on suitably small input domains, i.e., 
	$(R^f_U)^{-1} = R^\iota_U$, on domains, where the associated maps are well defined.
\end{lemma}
\subsection{The maps of \eqref{eq:def_Rf} define a retraction}
We confirm that $R^f_U$ of \eqref{eq:def_Rf} features the three properties mentioned in the introduction:
It maps into the symplectic Stiefel manifold,
$R^f_U(D)\in \SpStnp$, by construction.
Moreover, for the special input $D=0$, we obtain $K=\Delta^T J_n\Delta= J_p$ in \eqref{eq:DeltaTJDelta}, which entails that 
$\sqrt{-KJ_p} = \sqrt{I_{2p}}= I_{2p}$, so that the associated `symplectifier' is $G_f=I_{2p}$ and
\[
	R^f_U(0)= \left(U\exp_m(U^+0) + (I-UU^+)0\right)I_{2p} = U.
\]
For computing the differential, we make a series expansion of $t\mapsto R^f_U(tD)$.
The matrix $K=K(t)$ from Alg. \ref{alg:Rf_U} corresponding to the input $tD$ reads
\[
	K(t) = J_p +  \mathcal{O}(t^2).
\]
By the Neumann series for the inverse of a matrix and the Taylor series for the square root function, the series expansion of
$G_f(t) = \left((-KJ_p)^{-\frac12}\right)^T$ is 
\[
	G_f(t) = \left( (I_{2p} + \mathcal{O}(t^2))^{-\frac12}\right)^T = I_{2p} + \mathcal{O}(t^2).
\]
This shows
\begin{align*}
	R^f_U(tD)&= \left(U\exp_m(t U^+D) + (I-UU^+)tD\right)G_f(t)\\
	&= \biggl(U\left(I_{2p} + t U^+D +\mathcal{O}(t^2)\right) \\
	& \qquad + (I-UU^+)tD\biggr) \left(I_{2p} + \mathcal{O}(t^2)\right)\\
	&=U + t(UU^+ D + (I-UU^+)D) + \mathcal{O}(t^2)\\
	&= U + tD  + \mathcal{O}(t^2).
\end{align*}
In particular
\[
	d(R^f_U)_0[D] = \frac{d}{dt}\big\vert_{t=0} R^f_U(tD) = D.
\]
This establishes the third property of a retraction.
Formula \eqref{eq:def_Rf} shows that $R^f_U(D)$ is smooth both in $D$ and in $U$.
(The matrix exponential, the principal matrix logarithm and the principal primary matrix square root functions are smooth when well-defined.)
The series expansion $-K(t)J_p = I_{2p} +  \mathcal{O}(t^2)$ also shows that for $|t|$ small, $-K(t)J_p$ is guaranteed to have all its eigenvalues 
close to $1$, which rules out the occurrence of eigenvalues with negative real part and guarantees that the matrix square root is well-defined.
We have established
\begin{theorem}
	The collection of maps $\{R^f_U\mid U\in\SpStnp\}$ defined by \eqref{eq:def_Rf} constitutes a retraction on $\SpStnp$.
\end{theorem}
With the next statement, we quantify the domain on which the retraction $R^f_U$ from \eqref{eq:def_Rf} is well-define.
In view of Alg. \ref{alg:Rf_U}, the key issue is whether the primary matrix square root is well-defined.
\begin{lemma}
	\label{lem:Rf_defined}
	The map $(U,D) \mapsto R^f_U(D)$ of \eqref{eq:def_Rf} is well-defined, if the spectral radius of
	\[
	\mathcal{E}:=\mathcal{E}(U,D):= D^TJ_n(I-UU^+)DJ_p
	\]
	is strictly smaller than 1.
	A sufficient (not necessary) criterion for this to hold is
	\[
		\|D\|_2\|U\|_2< 1.
	\]
\end{lemma}
\begin{proof}
	The matrix $K=\Delta^TJ_n\Delta$ from Alg. \ref{alg:Rf_U}
	can be written as $K= J_p + D^TJ_n(I-UU^+)D$.
	We require the primary principal square root of
	\[
	-KJ_p = I_{2p} - D^TJ_n(I-UU^+)DJ_p =: I_{2p} -\mathcal{E}.
	\]
	The eigenvalues of $I_{2p} -\mathcal{E}$ are $1-\lambda_j$, 
	where $\lambda_1,\ldots, \lambda_{2p}$ are the eigenvalues of $\mathcal{E}$. By assumption, the spectral radius $r_\sigma(\mathcal{E})$ is smaller than one, meaning that $\max_j |\lambda_j|<1$.
	Hence, all eigenvalues $1-\lambda_j$ are strictly contained in the unit disk around $1\in\C$. In particular, they have positive real part, which guarantees the existence and uniqueness of the  primary principal square root, see \cite[Theorem 1.29]{Higham:2008:FM}.\\
	By \cite[Theorem 2.6]{stewart_sun:1990}, 
	\[
	r_\sigma(\mathcal{E}) \leq \|\mathcal{E}\|_2 \leq \|D\|_2\|I-UU^+\|_2 \|D\|_2.
	\]
	By the projector norm identity \cite{Szyld2006},
	$\|I-UU^+\|_2 = \|UU^+\|_2\leq \|U\|_2^2.$ 
\end{proof}
\section{Numerical demonstration}
\label{sec:numex}
We include two preliminary experiments to expose numerical properties of the polar-light retraction. In the first one, we consider numerical accuracy and computation time for the symplectic polar-light retraction and its inverse under changes in the dimension $p$ of $\SpStnp$. In the second experiment, we conduct a numerical competition, including  alternative retractions reported in the literature. The implementation relies on the functions listed in Section \ref{sec:matlab}.
\subsection{Numerical behavior of the symplectic polar-light retraction}
We evaluate $R^f_U(D)$ of \eqref{eq:def_Rf} and $R^\iota_U(\tilde U)$ of \eqref{eq:def_Ri} for (pseudo) random inputs $U\in \SpStnp$,
$D\in T_U\SpStnp$ for dimensions $n=2000$ and $p\in\{200,400,600,800,1000\}$.
For this experiment, random symplectic Stiefel data is created as follows. We compute a QR-decomposition $M_{\C}=QR$ of a (pseudo) random complex matrix $M_{\C}\in\C^{n\times n}$. This gives a unitary matrix $Q=\Re(Q) + i\Im(Q) =: X+iY$, which we use to form a symplectic matrix $M=\begin{pmatrix}
	X & Y\\
	-Y & X
\end{pmatrix}$, and eventually a symplectic Stiefel point $U=ME_{n,p}$.
A corresponding tangent vector is constructed from a Hamiltonian matrix $\Omega$ as in \eqref{eq:Ham_block} with random blocks by setting $D=M\Omega E_{n,p}$.
Note that this construction makes $U$ an ortho-symplectic Stiefel matrix, i.e., in addition to $U^+U=I_{2p}$, we have $U^TU=I_{2p}$.
%
%
The input $\tilde U$ to the inverse retraction $R^\iota_U(\tilde U)$ is the output of the forward retraction $\tilde U = R^f_U(D)$ (which of course makes the numerical errors accumulate).
Fig. \ref{fig:time_acc} reports the results. The upper plot shows the wallclock time of computing  $R^f_U(D)$ and $R^\iota_U(\tilde U)$ with the MatLab\cite{matlab} functions from Section \ref{sec:matlab} averaged over 10 runs. 
The lower plot reports the values
\begin{itemize}
	\item $\|R^f_U(D)^+ R^f_U(D) - I_{2p}\|_F$, 
	\item $\|(J_p U^+ R^\iota_U(\tilde U))^T - (J_p U^+R^\iota_U(\tilde  U))\|_F$, 
	\item $\|D -  R^\iota_U(R^f_U(D))\|_F$, 
	\item $\|\tilde U -  R^f_U(R^\iota_U(\tilde U))\|_F$,
\end{itemize}
which are numerical checks for
\begin{itemize}
	\item $R^f_U(D)$ producing a valid point on $\SpStnp$, 
	\item $R^\iota_U(\tilde U)$ producing a valid tangent on $T_U\SpStnp$, 
	\item $R^\iota_U\circ R^f_U=\mathrm{id}_{\SpStnp}$,
	\item $R^f_U\circ R^\iota_U=\mathrm{id}_{T_U\SpStnp}$,
\end{itemize}
respectively.
The figure confirms the numerical feasibility of the proposed retractions maps both in terms of computation time and accuracy.
The computational effort for evaluating \eqref{alg:Rf_U} and \eqref{alg:Ri_U} is essentially the same. Yet, the reverse retraction suffers less from numerical round-off errors than the forward retraction.
\begin{figure}[htbp]
	\centerline{\includegraphics[width=0.5\textwidth]{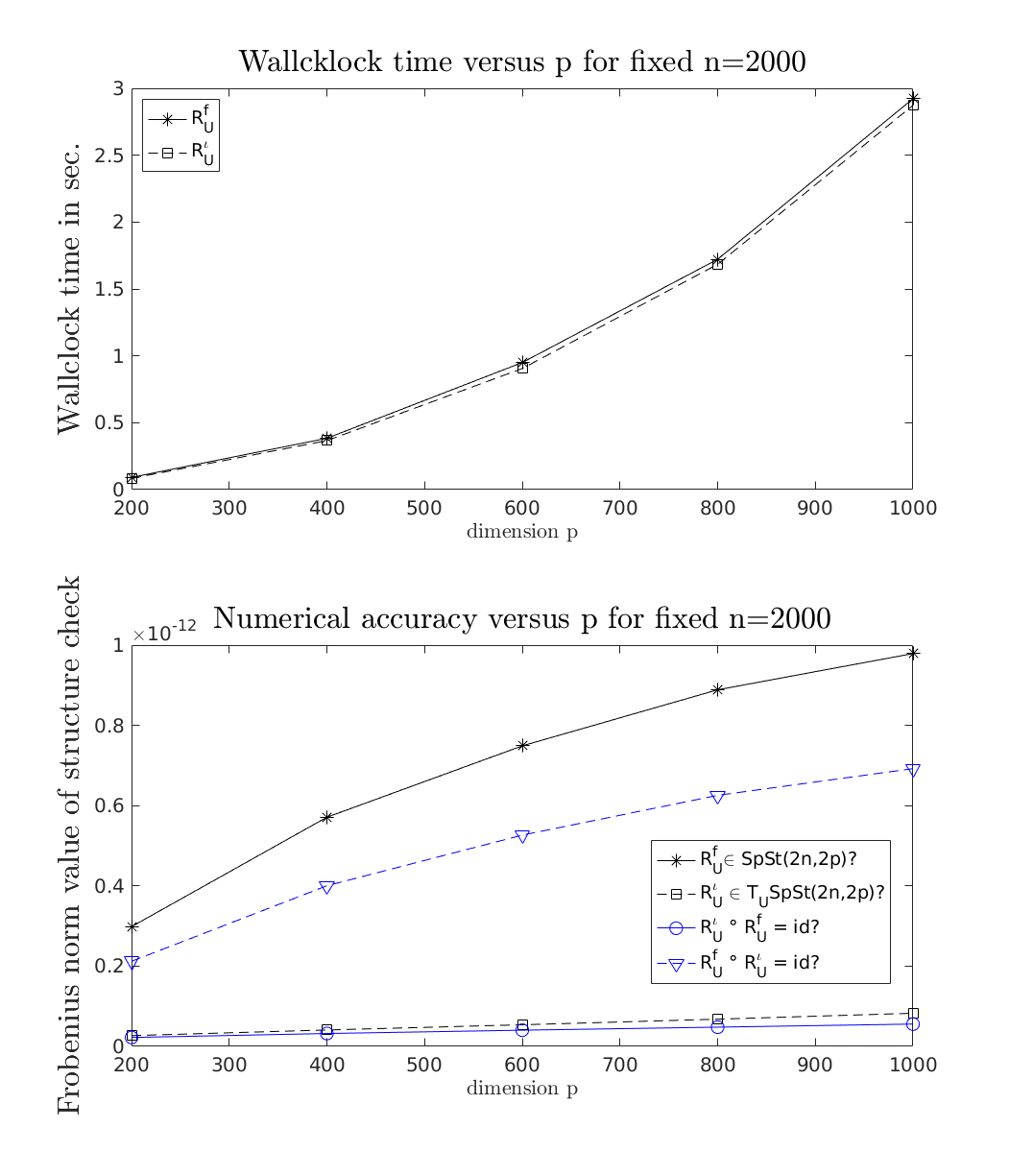}}
	\caption{Wallclock time (upper plot) and numerical accuracy (lower plot) associated with the forward and reverse symplectic polar-light retraction versus increasing block dimension $p$.}
	\label{fig:time_acc}
\end{figure}

\subsection{Comparison with state-of-the-art retraction}
In this subsection, we juxtapose the numerical features of the symplectic polar-light retraction with alternative retractions reported in the literature.
In the competition, we include
\begin{enumerate}
	\item the Riemann exponential (in reduced form) according to \cite[eq. (3.19)]{BendokatZimmermann:2021},
	\item the pseudo-Riemann exponential according to \cite[eq. (3.11)]{BendokatZimmermann:2021},
	\item The Cayley retraction that features in
  \cite[Prop. 5.5]{GaoSonAbsilStykel2020symplecticoptimization},	\cite[Prop. 5.2]{BendokatZimmermann:2021},
	\cite[eq. (22)]{OviedoHerrera:2023}, where the implementation follows the reduced form of the latter two references,
	\item the quasi-geodesic retraction of \cite[Lemma 5.1]{GaoSonAbsilStykel2020symplecticoptimization},
	\item the symplectic polar-light retraction of \eqref{eq:def_Rf}.
\end{enumerate}
As in \cite[Section 6.1]{BendokatZimmermann:2021}, a (pseudo) random point on $\SpStnp)$ is generated via $U = \text{cay}(\Omega)E$, with a Hamiltonian matrix $\Omega \in \sptn$ of unit Frobenius norm $\|\Omega\|_F=1$. Then a (pseudo) random tangent vector $\Delta \in T_U\SpStnp$, also scaled to $\|\Delta\|_F = 1$ is constructed.
The dimensions are set to $2n=4000$, $2p=1600$. 
For each retraction $R$ under consideration, we compute $R_U(D)$
for ten such random inputs and report the average effort in terms of the wallclock time and the average numerical feasibility in terms of
 $\|R_U(D)^+ R_U(D) - I_{2n}\|_F$. The latter quantifies, if the retraction output is indeed on the symplectic Stiefel manifold.
The results are reported in Table \ref{tab:performance_comp}.
In terms of efficiency, the polar-light retraction is second to the Cayley retraction, but being about three times slower. It is 1.5 to 2 times faster than the next best competitor. The numerical error is roughly one order of magnitude larger than for its competitors. This is likely due to the usage of a generic matrix square root function.
\begin{table}[htbp]
\caption{Performance comparison of known retractions on $\SpStnp$}
\begin{center}
\begin{tabular}{|l|c|c|}
\hline
\textbf{\textit{Retraction}} & \textbf{\textit{Time}} (sec.)&   $\|R_U(D)^+ R_U(D) - I_{2n}\|_F$\\
\hline
1) Riem. Exp & $26.20$& $4.42$e-14\\
\hline
2) p-Riem. Exp & $3.93$& $4.08$e-14\\
\hline
3) Cayley     & $0.74$& $7.00$e-14\\
\hline
4) quasi-geo &$4.24$& $5.08$e-14\\
\hline
5) polar-light & $2.32$& $9.10$e-13\\
\hline
\end{tabular}
\label{tab:performance_comp}
\end{center}
\end{table}
%
%

%
%
%
%
\section{Conclusions and future work}
\label{sec:conclusions}
By transferring the ideas of \cite{JensenZimmermann:2026} to the symplectic Stiefel manifold, we constructed a new retraction that features a closed form inverse.
To the best of our knowledge, prior to this work, the only other symplectic Stiefel retraction with this property was the Calyely retraction, where the closed-form inverse was found in \cite{BendokatZimmermann:2021}.

As with the Cayley retraction, for data on $\SpStnp$, only matrix function (such as $\exp_m,\log_m, \mathrm{Cay}, \mathrm{Cay}^{-1}, \sqrt{(\cdot)}$) that act on input matrices of dimensions $(2p\times 2p)$ are required.

The new retraction proved to be feasible and efficient, but about three times slower than the Cayley retraction. 
In a basic implementation, the numerical accuracy of the forward retraction is one order of magnitude lower than that of its competitors. Surprisingly, Fig. \ref{fig:time_acc} suggests that the numerical accuracy of the {\em reversed} retraction is on par with that of the other maps listed in Table \ref{tab:performance_comp}, which suggests to revisit implementation issues.

Additional properties such as the retraction order will be subject to future investigations.

%

\section{MatLab code}
\label{sec:matlab}
In this section, we list  MatLab \cite{matlab} code for a basic implementation of Algs. \ref{alg:Rf_U}, \ref{alg:Ri_U} that was used to produce the numerical results in Section \ref{sec:numex}.
The reader should be aware that the code uses MatLab's standard built-in function for computing principal matrix square roots,
which do not exploit the skew-Hamiltonian structure.
Performance and stability enhancements can be expected with a dedicated implementation. 
\paragraph{Algorithm \ref{alg:Rf_U} in MatLab syntax}
\begin{small}
	\begin{verbatim}
		function [RetU_D, Gf] ...
		= SpSt_PL_skewHam_ret(U, D, Jn, Jp)		
		[N, K] = size(U);    
		n      = floor(N/2);
		p      = floor(K/2);		
		% create sparse identity matrix
		I2p=speye(2*p);	
		% A = U^+ * D in sp(2p)
		Up = Jp'*U'*Jn;
		A  = Up*D;
		% Cayley gives symplectic 2px2p matrix
		S = Cayley(A);
		% update U
		UpD = U*(S-A) + D; 
		% Compute Y = -Delta^T J_n Delta*J
		% we build directly Y-transpose
		YT = I2p - Jp*(D'*(Jn*D)) + Jp*(A'*(Jp*A));
		Gf = inv(sqrtm(YT));
		%ensure skew-Hamiltonian structure
		Gf = (0.5)*Jp*((Jp*Gf)' - Jp*Gf);
		% Retracted matrix, Gf is skew-Hamiltonian
		RetU_D = (UpD)*Gf;
		end		
	\end{verbatim}
\end{small}
\paragraph{Algorithm \ref{alg:Ri_U} in MatLab syntax}
\begin{small}
\begin{verbatim}
function [TanU_U1, Gr] ...
   = SpSt_PL_invret(U, U1, Jn, Jp)
% Alg. 2 in MatLab syntax
[N, K]  = size(U);   
n       = floor(N/2);
p       = floor(K/2);
% form skew-symmetric kernel
Uplus   = Jp'*U'*Jn;
UplusU1 = (Uplus*U1);
% compute helper matrix Y-transpose
YT = -Jp*UplusU1'*Jp*UplusU1;
% inverse of primary square root of Y
Gr  = inv(sqrtm(YT));
%ensure skew-Hamiltonian structure
Gr = (0.5)*Jp*((Jp*Gr)' - Jp*Gr);
% form "symplectifier"
Sr = UplusU1*Gr;
% inverse Cayley gives matrix in sp(2p)
A = Cayley_inv(Sr);
%ensure Hamiltonian structure
A = (-0.5)*Jp*((Jp*A)' + Jp*A);
% Compute U*A + (I-U U^+) U1*Gr
TanU_U1 = U*(A-Sr) + U1*Gr;
end
\end{verbatim}
\end{small}
\paragraph{Auxiliary functions}
The structure matrices are produced with
\begin{small}
	\begin{verbatim}
Jp = spdiags([-ones(2*p,1), ones(2*p,1)],...
             [-p p], 2*p, 2*p);
	\end{verbatim}
\end{small}
Likewise for $J_n$.

The Cayley transformations are implemented as follows:
\begin{small}
	\begin{verbatim}
function [Cay] =  Cayley(X)
p =size(X,1);
% diagonal indices
diag_pp = find(eye(size(X)));
% form I-0.5X
Xminus = -0.5*X;  
Xminus(diag_pp) = Xminus(diag_pp) + 1.0;
% form I+0.5X
Xplus  = 0.5*X;
Xplus(diag_pp) = Xplus(diag_pp) + 1.0;
Cay = linsolve(Xminus, Xplus);
return;
end

function [Cay_inv] = Cayley_inv(Y)
p = size(Y,1);
% diagonal indices
diag_pp = find(eye(size(Y)));
% build 2*(Y-I)
Yminus = 2*Y;  
Yminus(diag_pp) = Yminus(diag_pp) - 2.0;
%
Yplus  = Y;   
Yplus(diag_pp)  = Yplus(diag_pp) + 1.0;
Cay_inv = linsolve(Yplus, Yminus);
return;
end

	\end{verbatim}
\end{small}

\bibliographystyle{IEEEtran}
\bibliography{IEEEabrv,IEEEexample}

\end{document}